\documentclass[12pt]{amsart}

\usepackage{graphicx}
\usepackage{pst-all}
\usepackage{amsmath}
\usepackage{amsfonts}
\usepackage{amssymb}
\usepackage{amsthm}
\usepackage{amscd}
\usepackage{url}
\usepackage{subfigure}
\input xy
\xyoption{all}
\usepackage{nicefrac}

\newtheorem{theorem}{Theorem}
\newtheorem*{theorem*}{Main Theorem}
\newtheorem*{theorem**}{Annulus Theorem}
\newtheorem{corollary}[theorem]{Corollary}

\newtheorem{proposition}[theorem]{Proposition}
\theoremstyle{definition}
\newtheorem{definition}[theorem]{Definition}
\newtheorem{example}[theorem]{Example}
\newtheorem*{remark}{Remark}

\begin{document}

\author{H\'ector Barge}
\address{E.T.S. Ingenieros inform\'{a}ticos. Universidad Polit\'{e}cnica de Madrid. 28660 Madrid (Espa{\~{n}}a)}
\email{h.barge@upm.es}

\author{Jos\'e M.R. Sanjurjo}
\address{Facultad de Ciencias Matem{\'a}ticas and Instituto de Matem\'atica Interdisciplinar (IMI). Universidad Complutense de Madrid. 28040 Madrid (Espa{\~{n}}a)}
\email{jose\_sanjurjo@mat.ucm.es}

\thanks{The authors are supported by the Spanish Ministerio de Ciencia e Innovaci{\'{o}}n (grant PGC2018-098321-B-I00)}
\title[Higher dimensional topology and generalized Hopf bifurcations]{Higher dimensional topology and generalized Hopf bifurcations for discrete dynamical systems}

\begin{abstract}
In this paper we study generalized Poincar\'e-Andronov-Hopf bifurcations of discrete dynamical systems. We prove a general result for attractors in $n$-dimensional manifolds satisfying some suitable conditions. This result allows us to obtain sharper Hopf bifurcation theorems for fixed points in the general case and other attractors in low dimensional manifolds. Topological techniques based on the notion of concentricity of manifolds play a substantial role in the paper.
\end{abstract}

\maketitle

\begin{center}
{\em  Affectionally dedicated to Mar\'ia Jes\'us Chasco on the ocassion of her 65th birthday}
\end{center}

\section{Introduction and preliminaries}

\subsection{Introduction}

In this paper we study families of homeomorphisms $f_\lambda:M\longrightarrow M$ parametrized by the unit interval and defined on $n$-dimensional manifolds. We assume that $f_0$ has an attractor that loses its stability and becomes a repeller for every $\lambda>0$. This phenomenon is known as \emph{generalized Poincar\'e-Andronov-Hopf bifurcation} or, shorlty, \emph{Hopf bifurcation}. 

\begin{definition}
Let $M$  be an $n$-dimensional manifold and $f_\lambda:M\longrightarrow M$ be a family of homeomorphisms continuosly depending on a parameter $\lambda\in[0,1]$. Suppose that $K\subset M$ is an attractor of $f_0$. We say that $K$ undergoes a \emph{Hopf bifurcation} at $\lambda=0$ if $K$ is a repeller of $f_\lambda$ for every $\lambda>0$.
\end{definition}

The study of this kind of bifurcations was originated in the works of Poincar\'e and was continued by Andronov and Hopf. The most famous Hopf bifurcation result refers to the development of periodic orbits from a stable fixed point of a  family of flows induced by a family of ordinary differential equations defined on the plane. There is an analogous result for parametrized families of diffeomorphisms due to Naimark \cite{Naim}, Sacker\cite{Sack} and Ruelle and Takens \cite{RuTa}. For further information about Hopf bifurcations see the book on bifurcation theory by Mardsen and McCracken \cite{MarMc}. See also \cite{SeFlo} for information about the foundations of bifurcation theory.

In this paper we study Hopf bifurcations of general attractors of discrete dynamical systems. In particular, we use homotopical techniques to study the topology of the attractors generated by Hopf bifurcations. We use a form of homotopy theory known as \emph{shape theory} or \emph{Borsuk's homotopy theory}. This theory  has proved to be very useful to study attractors of flows \cite{GuSe, Hast, KapRod, Rob2, SGRAC, Sanin, SanMul}. However in the discrete case this theory has limitations produced by the lack of natural homotopies provided by the dynamical system \cite{RSG, GabSd}. The theory of shape has been used to study bifurcations in the case of continuous dynamical systems \cite{BaSanRAC, BaSanTappl, Barge2020, GLS, Sanjhopf} and it has also been used to study Hopf bifurcations of fixed poins of planar homeomorphisms \cite{BGS}.

 The main result of this paper is the following:

\begin{theorem*}\label{main}
 Let $M$  be an $n$-dimensional manifold and $f_\lambda:M\longrightarrow M$ be a family of homeomorphisms continuosly depending on a parameter $\lambda\in[0,1]$. Suppose that 
\begin{enumerate} 
\item  $K$ is an attractor of $f_0$.
\item $K$ is a topological spine of some compact $n$-dimensional manifold with boundary $N\subset M$ that satisfies the concentric rigidity property. 
\item $K$ undergoes a Hopf bifurcation at $\lambda=0$. 
\end{enumerate}
 Then, there exists $\lambda_0>0$ such that for every $\lambda$ with $0<\lambda<\lambda_0$ there exists an attractor $K_\lambda$ that has the Borsuk's homotopy type (shape) of $\partial N$ and converges to $K$ upper semicontinuosly as $\lambda$ tends to $0$.
\end{theorem*}

As applications of this result we obtain sharper Hopf bifurcation theorems for:
\begin{itemize}
\item Fixed points in $n$-manifolds.
\item Arbitrary  attractors in connected $2$-manifolds.
\item Tame knots in $3$-manifolds.
\item Tame two sided closed $2$-manifolds in $3$-manifolds.
\item Tame $3$-manifolds with boundary in $3$-manifolds.
\item We also obtain other kind of bifurcations involving lower-dimensional attracting spheres.
\end{itemize}

\subsection{Preliminaries}
\subsubsection{Attractors} We recall some standard definitions of dynamical systems. Let $X$ be a locally compact metric space and $f:X\longrightarrow X$ a homeomorphism. Given a subset $Y\subset X$ the $\omega$-\emph{limit} set of $Y$ is defined as
\[
\omega(Y)=\bigcap_{n>0 }\left(\overline{\bigcup_{k\geq n} f^k(Y)}\right).
\] 
The set $\omega(Y)$ is a closed, invariant set. We say that a compactum $K$ is an \emph{attractor} if it possesses a  neighborhood $U$ such that $K=\omega(U)$. The subset $\mathcal{A}(K)=\bigcup_{n\geq 0} f^{-n}(U)$  is called \emph{basin of attraction} of $K$ and is an open invariant subset of $X$.  Attractors are invariant sets and satisfy the following stability property: for every compactum $P\subset \mathcal{A}(K)$ and every neighborhood $V$ of $K$ there exists $n_0\geq 0$ such that $f^{n}(P)\subset V$ for every $n\geq n_0$.  A useful way of characterizing attractors is by means of the so-called traping regions. A \emph{trapping region} $W\subset X$ is a compactum with non-empty interior such that $f(W)\subset\mathring{W}$. Here $\mathring{W}$ denotes the topological interior of $W$. Attractors are determined by trapping regions in the sense that if $K$ is an attractor it possesses a trapping region $W$ and  
\[
K=\bigcap_{n\geq 0}f^n(W)
\]
is an attractor contained in $W$. Notice that trapping regions are robust in the sense that if $W$ is a trapping region for a homeomorphism $f$ and $g$ is another homeomorphism sufficiently close to $f$ then $W$ is also a trapping region for $g$.

A \emph{repeller} $K\subset X$ is a compact set which is an attractor for $f^{-1}$ and we call \emph{basin of repulsion} of $K$, and denote it by $\mathcal{R}(K)$, to the basin of attraction of $K$ for $f^{-1}$. For more information about attractors see \cite{AKHuKe}.

\subsubsection{Bicollars} Let $M$ be an $n$-dimensional manifold. An $(n-1)$-manifold $N\subset M$ is said to be \emph{bicollared} if there exists an embedding $h:N\times[-1,1]\longrightarrow M$ such that $h(N\times\{0\})=N$. Notice that by the Theorem of invariance of domain  $C=h(N\times[-1,1])$ is a neighborhood of $N$ in $M$.

\subsubsection{Triangulations and tame subsets}

A \emph{triangulation} of a topological space $X$ is a pair $(\mathcal{S},h)$ where $\mathcal{S}$ is a locally finite simplicial complex and $h:X\longrightarrow|\mathcal{S}|$ is a homeomorphism between $X$ and the geometric realization of $\mathcal{S}$. 

It is known that every manifold of dimension not larger than three admits a triangulation. This result was proved by Rad\'o in dimension two and by Moise in dimension three. Moreover, every  triangulation of a manifold of dimension $n\leq 3$ is combinatorial. That is, the star of every vertex of the triangulation is combinatorially equivalent to an $n$-dimensional simplex. The book by Moise \cite{Moiseb} contains the proofs of these deep results.

Let $M$ be a $3$-manifold. We say that a compactum $K\subset M$ is \emph{tame} if there exists a triangulation $(\mathcal{S},h)$ of $M$ and a subcomplex $\mathcal{T}\subset\mathcal{S}$ such that $(\mathcal{T},h|_K)$ is a triangulation of $K$. A useful criterion for tameness is \cite[Lemma~5]{SGNLA} which establishes that if $M$ is a $3$-manifold with boundary
\begin{enumerate}
\item A closed $2$-manifold $S\subset{\rm int}(M)$ is tame in ${\rm int}(M)$ if it is bicollared.

\item A compact $3$-manifold with boundary $P\subset{\rm int}(M)$ is tame in ${\rm int}(M)$ if $\partial P$ is bicollared.
\end{enumerate}

\subsubsection{Algebraic topology and Borsuk's homotopy theory}

In this paper we shall make use of some concepts from Algebraic Topology including Alexander and Lefschetz duality theorems. We denote by $H_*$ and $H^*$ the singular homology and cohomology functors respectively and by $\check{H}^*$ the \v Cech cohomology functor. Some standard references covering this material are the books by Hatcher \cite{Hat} and Munkres \cite{Munk}. 

We shall make use of a form of homotopy theory known as \emph{Borsuk's homotopy theory} or \emph{shape theory} that has proved to be useful in the study of attractors of dynamical systems. A  detailed treatment about the theory of shape can be found in the books by Borsuk \cite{Bormono}, Mardesi\v{c} and Segal \cite{MarSe} and Dydak and Segal \cite{DySe}. 

Suppose that $X$ and $Y$ are compact metric spaces. We shall make use of the following properties:

\begin{enumerate}

\item If $X$ and $Y$ have the same homotopy type then $X$ and $Y$ have the same shape.

\item If $X$ and $Y$ are ANRs then $X$ and $Y$ have the same homotopy type if and only if $X$ and $Y$ have the same shape.

\item If $X$ and $Y$ have the same shape then $X$ and $Y$ have isomorphic \v Cech cohomology groups. 
\end{enumerate}

Notice that polyhedra, CW-complexes and manifolds are examples of ANRs.

We shall use the following criterion for shape equivalence (see \cite[Theorem~6]{SGRAC}):

 Let $\{N_k\}_{k\geq 0}$ be a sequence of metric compacta such that $N_{k+1}\subset N_k$ for every $k\geq 0$ and $K=\bigcap_{k\geq 0} N_k$. If for every $k\geq 0$ the inclusion $i_k:N_k\hookrightarrow N_0$ is a shape equivalence then so is the inclusion $i:K\hookrightarrow N_0$.
 
In addition we shall also make use of Borsuk's classification of the shape of plane continua that ensures that two continua contained in $\mathbb{R}^2$ have the same shape if and only if they separate $\mathbb{R}^2$ into the same number of connected components (\cite[Theorem~9.1]{BorCon}). 

\subsection{Outline of the paper}

The paper is structured as follows: in Section~\ref{sec:2} we introduce the definition of concentric manifolds and the concentric rigidity property. We show that the concentric rigidity property is satisfied by the $n$-dimensional closed ball (Proposition~\ref{con:ball}) and by any compact manifold with boundary of dimension two and three (Proposition~\ref{prop:con23}). In Section~\ref{sec:3} we introduce topological spines and give a characterization of them in terms of bases of neighborhoods (Proposition~\ref{prop:con}). In addition we deduce that if $K$ is a topological spine of some compact $n$-manifold with boundary $N$ then the inclusion $i:K\hookrightarrow N$ is a shape equivalence (Corollary~\ref{cor:shape}). We also present some examples of attractors in $\mathbb{R}^3$ that are not spines of any compact $3$-manifold with boundary contained in $\mathbb{R}^3$ (Example~1 and Example~2). Section~\ref{sec:4} is devoted to present some applications of the Main Theorem. We prove that if a fixed point undergoes a Hopf bifurcation it expels an attractor with the Borsuk'shomotopy type of the $(n-1)$-dimensional sphere, where $n$ is the dimension of the phase space (Theorem~\ref{thm:fixed}). We see that if an attracting proper subcontinuum of a connected $2$-manifold undergoes a Hopf bifurcation it expels an attractor with the Borsuk's homotopy type of a finite disjoint union of circles (Theorem~\ref{thm:2man}). The number of components of this attractor is determined whenever the phase space is either the plane or the $2$-sphere. We also study Hopf bifurcations of tame manifolds contained in $3$-manifolds.  We show that if an attracting tame knot in a $3$-manifold undergoes a Hopf bifurcation it expels an attractor with the Borsuk's homotopy type of either the torus or the Klein bottle (Theorem~\ref{thm:knot}). We see that if an attracting tame two sided closed $2$-manifold $K$ in a $3$-manifold undergoes a Hopf bifurcation it expels an attractor with the Borsuk's homotopy type of the disjoint union of two copies of $K$. We also prove that if a tame compact $3$-manifold with boundary  contained in a $3$-manifold undergoes a Hopf bifurcation it expels an attractor with the Borsuk's homotopy  type of its boundary. In Section~\ref{sec:5} we present some general results about attractors. In Section~\ref{sec:6} we present the proof of the Main Theorem. Finally, in Section~\ref{sec:7} we study Hopf bifurcations that occur inside invariant submanfolds, that is, for the restriction of the family of homeomorphisms to an invariant submanifold. In particular, we see that if a fixed point is contained in some invariant $n$-dimensional manifold and undergoes a Hopf bifurcation inside this invariant manifold, then it expels an attractor that has the Borsuk's homotopy type of $S^{n-1}$.

\section{Concentric manifolds and the concentric rigidity property}\label{sec:2}

In this section we recall the concept of concentricity and introduce the concentric rigidity property. We also see that the closed $n$-dimensional ball and every compact two and three manifold with boundary satisfy this property.
\begin{definition}
Let $M$ and $N$ be compact $n$-dimensional manifolds with boundary such that $N\subset {\rm int} (M)$. We say that $M$ and $N$  are \emph{concentric} if $M\setminus{\rm int}(N)$ is homeomorphic to $\partial M\times[0,1]$.
\end{definition}

\begin{remark}
From the definition it follows that if $N\subset {\rm int}(M)$ is concentric with $M$, then $M$ is obtained from $N$ by attaching an exterior collar. This ensures that $N$ is homeomorphic to $M$ (see for instance the proof of \cite[Proposition~3.2]{Hat}). It also follows that $\partial N$ is bicollared. To see this observe that since $N$ is a compact manifold with boundary, $\partial N$ has a collar $C$ contained in $N$. As a consequence $C\cup (M\setminus {\rm int} (N))$ is a bicollar of $\partial N$. 
\end{remark}

Motivated by the work by Edwards about concentricity of $3$-manifolds \cite{Ed1,Ed2} we introduce the following definition.

\begin{definition}\label{def:con}
Let $M$ be a compact $n$-dimensional manifold with boundary. We say that $M$ has  the \emph{concentric rigidity property} if given a pair of compact $n$-manifolds with boundary $M_0$ and $M_1$ that satisfy:
\begin{enumerate}
\item   $M_0\subset{\rm int}(M_1)\subset M_1\subset {\rm int}(M)$.

\item $\partial M_1$  is bicollared and homeomorphic to $\partial M$.

\item $M_0$ is concentric with $M$.
\end{enumerate} 
Then $M_1$ is concentric with both $M_0$ and $M$. 
\end{definition}

\begin{remark}
It follows from the definition that the concentric rigidity property is a topological property.
\end{remark}

The following result shows that the $n$-dimensional closed ball satisfies the concentric rigidity property. To prove this result we need a powerful result known as the annulus Theorem.

\begin{theorem**}
Let $B'$ be an $n$-cell contained in the interior of an $n$-cell $B$. Suppose that $\partial B'$ is bicollared. Then $B\setminus {\rm int}(B')$ is homeomorphic to $S^{n-1}\times[0,1]$. 
\end{theorem**}

We recall that an $n$-cell is a topological space homeomorphic to the $n$-dimensional closed ball. The annulus Theorem was proved by Rad\'o  in dimension $2$ \cite{Rado}, by Moise in dimension $3$ \cite{Moise}, by Quinn in dimension $4$ \cite{Quinn} and  by Kirby \cite{Kirby} for $n\geq 5$.  See also \cite[Theorem~7.5.3, pg. 374]{DavermanVenema}.

\begin{proposition}\label{con:ball}
The $n$-dimensional closed ball satisfies the concentric rigidity property. 
\end{proposition}

\begin{proof}

Let $M$ be the $n$-dimensional closed ball. Suppose that $M_0$ and $M_1$ are compact $n$-manifolds with boundary satisfying (1), (2) and (3) from Definition~\ref{def:con}.  Since $\partial M_1$ is a bicollared $(n-1)$-sphere contained in ${\rm int}(M)$, the generalized Sch\"{o}enflies Theorem \cite[Theorem~2.4.8, pg. 62]{DavermanVenema} ensures that $\partial M_1$ is the boundary of a topological  closed $n$-ball $B\subset {\rm int}(M)$. In addition, $\partial M_1$ decomposes $M$ into two connected components. We see that $B=M_1$.  To see this notice that
\[
M\setminus\partial M_1={\rm int}(B)\cup (M\setminus B)= {\rm int}(M_1)\cup (M\setminus M_1).
\]
 Since $\partial M_1$ and $M$ are connected and $M_1$ is compact  ${\rm int} (M_1)$ is also connected. Hence ${\rm int}(M_1)$ must be one of the components of $M\setminus\partial M_1$. The other component must be $M\setminus M_1$. Reasoning in the same fashion ${\rm int}(B)$ must be a component of $M\setminus\partial M_1$ and, hence, it must coincide either with ${\rm int} (M_1)$ or with $M\setminus M_1$.  The second possibility is excluded since $M\setminus M_1$ contains $\partial M$ while $\partial M$ is not contained in ${\rm int}(B)$. 

It remains to see that the closed topological ball $M_1$ is concentric with both $M$ and $M_0$. Since $\partial M_1$ is bicollared and $M_1\subset {\rm int}(M)$ the concentricity of $M_1$ with $M$ follows from the annulus Theorem. On the other hand, since $M_0$ is concentric with $M$, it follows that $\partial M_0$ is bicollared in $M$ and, since ${\rm int}(M_1)$ is open in $M$ and $M_0\subset {\rm int}(M_1)$,  $\partial M_0$ is also bicollared in $M_1$. Therefore a new application of the annulus Theorem ensures that $M_0$ and $M_1$ are concentric. 
\end{proof}

We see that in addition to closed balls there are many other compact manifolds with boundary that satisfy the concentric rigidity property.

\begin{proposition}\label{prop:con23}
Every compact $n$-manifold with boundary satisfies the concentric rigidity property for $n=2$ and $n=3$. 
\end{proposition}

\begin{proof}
Let $M$ be a compact $n$-manifold with boundary and suppose that $M_0$ and $M_1$ are compact $n$-manifolds with boundary satisfying (1), (2) and (3) from Definition~\ref{def:con}. 

Suppose that $n=2$. Since $M$ and $M_0$ are concentric it follows that $M\setminus {\rm int}(M_0)$ is homeomorphic to a finite disjoint union of closed annuli. Let $A$ be a component of $M\setminus {\rm int}(M_0)$. Condition (1) ensures that the closed annulus $A$ contains exactly one component of $\partial M_1$ and that this component separates the two boundary components of $A$. Then invoking the $2$-dimensional annulus Theorem in each component of $M\setminus {\rm int}(M_0)$ ensures that $M_1$ is concentric with both $M_0$ and $M$.

Suppose that $n=3$. Since $M_0$ is concentric with $M$ it follows that $\partial M_0$ is bicollared in $M$ and, since ${\rm int}(M_1)$ is open in $M$, $\partial M_0$ is also bicollared in $M_1$. In addition, $\partial M_1$ is bicollared by assumption. Since $M_0$ has bicollared boundary in $M_1$ and $M_1$ has bicollared boundary in $M_2$, \cite[Lemma~5]{SGNLA} guarantees that $M_0$ is tame in ${\rm int}(M_1)$ and $M_1$ is tame in ${\rm int}(M)$. The result follows from \cite[Theorem~2]{Ed2}.
\end{proof}

\section{Topological spines}\label{sec:3}

In this section we introduce the concept of topological spine and study some of its properties. In addition we see some examples of attractors of discrete dynamical systems that are not topological spines.

\begin{definition}
Let $N$ be a compact $n$-manifold with boundary. We say that a compactum $K\subset {\rm int}(N)$ is a \emph{topological spine} of $N$ if $N\setminus K$ is homeomorphic to $\partial N\times[0,+\infty)$. 
\end{definition}

There are some other definitions of spine in the literature. See for instance \cite{HudZee} or \cite{GabSanjGl}. The following result gives a characterization of spines in terms of bases of neighborhoods.

\begin{proposition}\label{prop:con}
Suppose that $N$ is a compact $n$-manifold with boundary. A compactum $K\subset{\rm int}(N)$ is a topological spine of $N$ if and only if $K$ possesses a basis of neighborhoods $\{N_k\}_{k\geq 0}$ comprised of compact $n$-manifolds with boundary satisfying
\begin{enumerate}
\item[(1)] $N_0=N$.
\item[(2)] $N_{k+1}\subset{\rm int}(N_k)$.
\item[(3)] $N_{k}$ and $N_{k+1}$ are concentric.
\end{enumerate}
\end{proposition}

\begin{proof}
Since $K$ is a spine of $N$ there exists a homeomorphism $h:\partial N\times[0,+\infty)\longrightarrow N\setminus K$. For every $k\geq 0$ we define
\[
N_k=K\cup h(\partial N\times [k,+\infty)).
\]
 We see that $N_k$ is a compact neighborhood of $K$ for every $k>0$. The compactness of of $N_k$ follows from the fact that $N\setminus N_k=h(\partial N\times[0,k))$ is open in $N\setminus K$ and,  thus, in $N$. Hence, $N_k$ is closed in $N$ and, therefore, compact. $N_k$ is a neighorbood of $K$ since  $K\subset N_k\setminus h(\partial{N}\times\{k\})\subset N_k$ and $N_k\setminus h(\partial{N}\times\{k\})$ is open being the complement of the compact set $h(\partial N\times[0,k])$. Notice that this also ensures that $N_k$ is a $n$-manifold with boundary  $\partial N_k=h(\partial N\times\{k\})$. Hence $N_{k+1}\subset{\rm int}(N_{k})$. The concentricity of $N_k$ and $N_{k+1}$ follows from the fact that
 \[
N_{k}\setminus {\rm int}(N_{k+1})=h(\partial N\times[k,k+1]).
 \]
 
Conversely, suppose that $\{N_k\}$ is a basis of neighborhoods of $K$ comprised of compact $n$-manifolds with boundary satisfying (1), (2) and (3). First observe that
\[
N\setminus K=\bigcup_{k\geq 0}(N\setminus{\rm int} (N_{k+1}))
\]
We construct step by step a homeomorphism $N\setminus K$ to $\partial N\times[0,+\infty)$ as follows. Since $N$ is concentric with $N_1$ there exists a homeomorphism $h_1:N\setminus{\rm int}(N_1)\longrightarrow\partial N\times [0,1]$  such that $h_1(x)=(x,0)$ for every $x\in\partial N$. Reasoning in the same fashion, we can find a homeomorphism $\bar{h}_2:N_1\setminus{\rm int}(N_2)\longrightarrow \partial N_1\times[0,1]$ such that $\bar{h}_2(x)=(x,0)$ for every $x\in\partial N_1$. Let $H:\partial N_1\times[0,1]\to \partial N\times[1,2]$ be the homeomorphism given by $H(x,t)=(h_1|_{\partial N_1}(x),t+1)$. Then, the map $h_2:N\setminus{\rm int}(N_2)\longrightarrow \partial N\times[0,2]$ defined as
\[
h_2(x)=
\begin{cases}
h_1(x) & \mbox{if}\; x\in N\setminus N_1\\
H\circ \bar{h}_2(x) & \mbox{if}\; x\in N_1\setminus{\rm int}(N_2)
\end{cases}
\]
is a homeomorphism that extends $h_1$. If we continue this process we are able to construct, for each $k$, a homeomorphism $h_k:N\setminus{\rm int}(N_k)\longrightarrow \partial N\times[0,k]$ that extends $h_{k-1}$. We define $h_\infty:N\setminus K\longrightarrow \partial N\times[0,+\infty)$ as  $h_\infty(x)=h_k(x)$ if $x\in N\setminus N_{k}$. It is not difficult to see that $h_\infty$ is a continuous and open bijection and, hence, a homeomorphism.
\end{proof}

\begin{remark}
A direct consequence of Proposition~\ref{prop:con} is that the compactum $K$ is a spine of $N_k$ for every $k\geq 0$.
\end{remark}

\begin{corollary}\label{cor:shape}
Let $N$ be a compact $n$-manifold with boundary and suppose that $K$ is a topological spine of $N$. Then the inclusion $i:K\hookrightarrow N$ is a shape equivalence.
\end{corollary}

\begin{proof}
Since $K$ is a spine of $N$ it possesses a basis of neighborhoods $\{N_k\}$ comprised of compact $n$-manifolds with boundary satisfying conditions (1), (2) and (3) of  Proposition~\ref{prop:con}.  Since $N_{k+1}\subset {\rm int}(N_k)$ and $N_k$ and $N_{k+1}$ are concentric, it follows that there exists a strong deformation retraction from $N_k\setminus{\rm int}(N_{k+1})$ onto $\partial N_{k+1}$. This strong deformation retraction extends to a strong deformation retraction from $N_k$ onto $N_{k+1}$ by keeping all the points of ${\rm int} (N_{k+1})$ fixed. Therefore for every $k\geq 0$ the inclusion $i_{k+1,k}:N_{k+1}\hookrightarrow N_k$ is a homotopy equivalence. As a consequence, the inclusion $i_{k}:N_{k}\hookrightarrow N$ is also homotopy equivalence for every $k\geq 0$ and, hence, a shape equivalence. Then the inclusion $i: K\hookrightarrow N$ is a shape equivalence.
\end{proof}

It can be seen using smooth Lyapunov functions that every attractor of a smooth flow on a smooth $n$-dimensional manifold $M$ is a topological spine of some compact $n$-manifold with boundary $N$ with $N\subset M$. The following examples show that this does not hold in general for attractors of homeomorphisms. 

\begin{example}\label{ex:1}
Consider the homeomorphism $f:\mathbb{R}^3\longrightarrow\mathbb{R}^3$ that carries the standard solid torus $T$ in $\mathbb{R}^3$  into its interior as depicted in Figure~\ref{fig:solenoid}. The compactum
\[
K=\bigcap_{n\geq 0} f^{n}(T)
\]
is an attractor of $f$ known as the \emph{dyadic solenoid}. We see that $K$ cannot be a topological spine of any compact $d$-dimensional manifold with boundary $N$. Suppose, arguing by contradiction, that there exists a compact $d$-dimensional manifold $N$ having $K$ as a topological spine. Then Corollary~\ref{cor:shape} ensures that $i:K\hookrightarrow N$ is a shape equivalence. Since shape equivalences induce isomorphisms in \v Cech cohomology it follows that 
\[
\check{H}^*(K;\mathbb{Z})\cong\check{H}^*(N;\mathbb{Z})\cong H^*(N;\mathbb{Z}).
\]
The last isomorphism holds because $N$ is a manifold with boundary. Since $N$ is compact $H^*(N;\mathbb{Z})$ is finitely generated in every dimension and, hence, so is $\check{H}^*(K;\mathbb{Z})$.  Consider the family $\{T_n\}_{n\geq 0}$ where $T_n=f^n(T)$. The construction and the continuity property of \v Cech cohomology ensure that
\[
\check{H}^1(K;\mathbb{Z})\cong\varinjlim H^1(T_n;\mathbb{Z})
\]
where the bonding maps are the homomorphisms induced by the inclusions $i_{n,n+1}: T_{n+1}\hookrightarrow T_n$. For each $n\geq 0$ the set  $T_n$ is a solid torus and hence $H^1(T_n;\mathbb{Z})$ is isomorphic to $\mathbb{Z}$. In addition, for each $n\geq 0$ the solid torus $T_{n+1}$ winds two times around  $T_n$ and, thus, all the bonding maps are the multiplication by $2$.  Since $\check{H}^1(K;\mathbb{Z})$ is finitely generated there must be some $n_0\geq 0$ and cohomology classes $w_1,\ldots, w_k\in H^1(T_{n_0};\mathbb{Z})$ whose images in $\check{H}^1(K;\mathbb{Z})$ generate $\check{H}^1(K;\mathbb{Z})$. Taking into account that every bonding map is injective it follows that every bonding map is also surjective for $n\geq n_0$. This is in contradiction with all the bonding maps being the multiplication by $2$. Therefore $K$ cannot be a spine of any compact $d$-manifold with boundary.

\begin{figure}
\begin{center}
\includegraphics[scale=1]{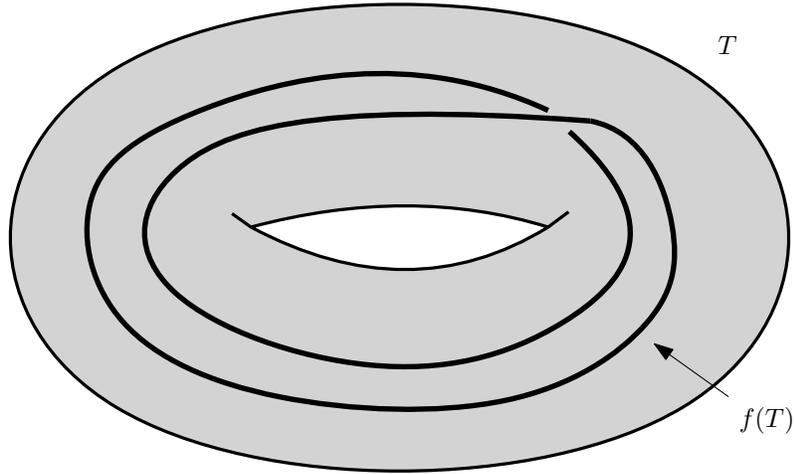}
\label{fig:solenoid}
\caption{Construction of the dyadic solenoid as an attractor of a homeomorphism of $\mathbb{R}^3$.}
\end{center}
\end{figure}

\end{example}

\begin{example}\label{ex:2}
In a similar way as in Example~\ref{ex:1} we consider a homeomorphism  $f:\mathbb{R}^3\longrightarrow\mathbb{R}^3$ that carries the standard solid torus $T$ in $\mathbb{R}^3$ into its interior as depicted in Figure~\ref{fig:whitehead}. The attractor $K$ determined by $T$  is known as the the \emph{Whitehead continuum}.  Since the solid torus $f(T)$ is contractible in $T$ it follows that each solid torus $f^{n+1}(T)$ is contractible in $f^n(T)$ and, as a consequence, $K$ has the Borsuk's homotopy type (shape) of a point. In particular, $\check{H}^*(K;\mathbb{Z})$ is finitely generated in every dimension. In spite of this $K$ is not a topological spine of any compact $3$-manifold with boundary $N\subset\mathbb{R}^3$ since if it were \cite[Theorem~4]{SGNLA} would ensure the existence of a flow in $\mathbb{R}^3$ having $K$ as an attractor. However it follows from \cite[Example~47]{SGball} that $K$ cannot be an attractor of a flow in $\mathbb{R}^3$.

\begin{figure}
\begin{center}
\includegraphics[scale=1]{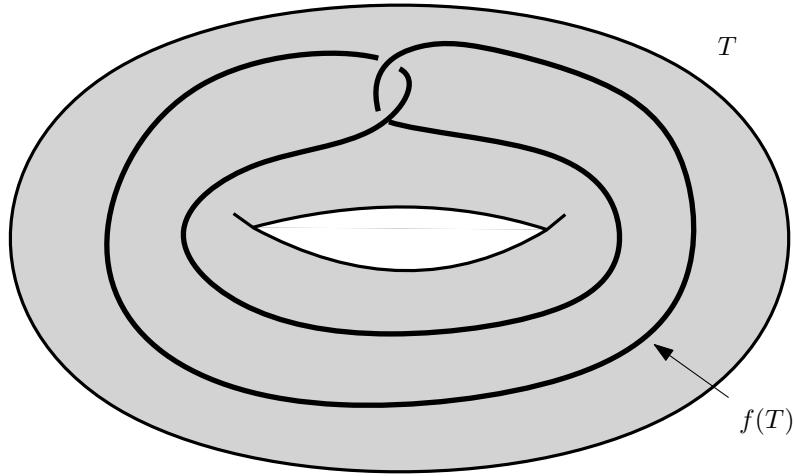}
\label{fig:whitehead}
\caption{Construction of the Whitehead continuum as an attractor of a homeomorphism of $\mathbb{R}^3$.}
\end{center}
\end{figure}

\end{example}

\section{Applications of the Main Theorem}\label{sec:4}

In this section we study Hopf bifurcations of fixed points and Hopf bifurcations of attractors in dimensions $2$ and $3$.

\subsection{Hopf bifurcations of fixed points}

\begin{theorem}\label{thm:fixed}
 Suppose that $p$ is an asymptotically stable fixed point for $f_0$ that undergoes a Hopf bifurcation. Then, there exists $\lambda_0>0$ such that for every $\lambda$ with $0<\lambda<\lambda_0$ there exists an attractor $K_\lambda$ which has the Borsuk's homotopy type (shape) of $S^{n-1}$ and converges to $\{p\}$ upper semicontinuosly as $\lambda$ tends to $0$.
\end{theorem}

\begin{proof}
This results follows from the Main Theorem taking into account that $p$ is a spine of an $n$-cell contained in $M$ and that $n$-cells have the concentric rigidity property.
\end{proof}

\begin{remark}
Theorem~\ref{thm:fixed} also holds if we substitute the fixed point by a cellular attractor.
\end{remark}

\subsection{Hopf bifurcations in $2$-manifolds}

\begin{theorem}\label{thm:2man}
Suppose that $M$ is a connected $2$-manifold and $K\subsetneq M$ is an attractor of $f_0$ that undergoes a Hopf bifurcation at $\lambda=0$. Then, there exists $\lambda_0>0$ such that for every $\lambda$ with $0<\lambda<\lambda_0$ there exists an attractor $K_\lambda$ that has the Borsuk's homotopy type (shape) of a disjoint union of a finite number of circles and converges to $K$ upper semicontinuosly as $\lambda$ tends to $0$. In particular, if $M$ is either $\mathbb{R}^2$ or $S^2$ and $K$ is connected the number of components of $K_\lambda$ coincides with the number of components of $M\setminus K$. 
\end{theorem}

\begin{proof}
Since $K$ is an attractor of $f_0$ \cite[Theorem~1]{RSG} ensures that $\check{H}^k(K;\mathbb{Z}_2)$ is finitely generated for $k=0,1$. Then \cite[Lemma~B.3]{SGth} together with \cite[Lemma~B.5]{SGth} ensure that $K$ has a basis of neighborhoods $\{N_k\}_{k\geq 0}$ satisfying  the conditions (2) and (3) of Proposition~\ref{prop:con} and, as a consequence, $K$ is a spine of $N=N_0$. Since compact $2$-manifolds with boundary satisfy the concenctric rigidity property by Proposition~\ref{prop:con23} the Main Theorem ensures the existence of $\lambda_0>0$ such that for $0<\lambda<\lambda_0$ there exists an attractor $K_\lambda$ with the shape of $\partial N$ converging to $K$ upper semicontinuously as $\lambda\to 0$. Since $N$ is a compact $2$-manifold with boundary it follows that $\partial N$ is a disjoint union of a finite number of topological circles.  If $M$ is either $\mathbb{R}^2$ or $S^2$ by Alexander duality $K$ separates $M$ into a finite number $k$ of components and if $K$ is connected Borsuk's characterization of the shape of plane continua ensures  that $K$ has the shape of a wedge of $k-1$ circles. Then, since $N$ is a compact $2$-manifold with boundary and $N$ and  $K$ have the same shape, $N$ has the homotopy type of a wedge of $k-1$ circles.  Taking into account that $N$ is contained in $M$ it follows that $N$ is homeomorphic to a $2$-sphere with the interiors of $k$ disjoint topological disks removed. Therefore $\partial N$ is homeomorphic to the disjoint union of $k$ circles.

\end{proof}

\subsection{Hopf bifurcations in $3$-manifolds}

Let $M$ be a $3$-manifold and $K\subset M$ a compactum. For each $x\in K$  consider the inclusion $i_x:(M,M\setminus K)\hookrightarrow (M,M\setminus\{x\})$. We say that $K$ \emph{preserves orientation} if there exists a homology class $\alpha_K\in H_3(M,M\setminus K;\mathbb{Z})$ such that  $i_{x*}(\alpha_K)$ is a generator of $H_3(M,M\setminus\{x\};\mathbb{Z})\cong\mathbb{Z}$ for each $x\in K$. Otherwise we say that $K$ \emph{reverses orientation}.

\begin{theorem}\label{thm:knot}
Suppose that $M$ is a $3$-manifold and $K\subset M$ is a tame knot that is an attractor for $f_0$. Suppose that $K$ undergoes a Hopf bifurcation at $\lambda=0$. Then, there exists $\lambda_0>0$ such that for every $\lambda$ with $0<\lambda<\lambda_0$ there exists an attractor $K_\lambda$ that has  the Borsuk's homotopy type (shape) of the torus if $K$ preserves orientation or the Borsuk's homotopy type of the Klein bottle if $K$ reverses orientation. Moreover, $K_\lambda$ converges to $K$ upper semicontinuosly as $\lambda$ tends to $0$. 
\end{theorem}

\begin{proof}
Since $K$ is tame we may assume without loss of generality that $M$ and $K$ are the geometric realizations of some simplicial complexes $\mathcal{S}$ and $\mathcal{T}$ with $\mathcal{T}$ a subcomplex of $\mathcal{S}$. Possibly after subdiving $\mathcal{S}$ we may assume that $\mathcal{T}$ is a full subcomplex of $\mathcal{S}$. That is, for any simplex $\sigma\in\mathcal{S}$ not contained in $K$ such that $\sigma\cap K\neq\emptyset$, the intersection $\sigma\cap K$ is either a vertex, an edge or a $2$-dimensional face of $\sigma$. Let $\mathcal{S}'$ be a derived subvidision of $\mathcal{S}$ near $\mathcal{K}$ (see \cite[pg. 32]{RouSan}). The simplicial neighborhood $N(\mathcal{T},\mathcal{S}')$ is defined as the union of the simplices of $\mathcal{S}'$ whose intersection with $K$ is non-empty. The Simplicial Neighborhood Theorem \cite[Theorem~3.11, pg. 34]{RouSan} ensures that the geometric realization $N=|N(\mathcal{T},\mathcal{S}')|$ is a compact $3$-manifold with boundary such that $K\subset{\rm int}(N)$. In addition, from \cite[Lemma~3.4.1, pg. 125 ]{DavermanVenema} $N$ is homeomorphic to the mapping cylinder with respect to a piecewise linear map $r:\partial N\longrightarrow K$. That is, $N$ is homeomorphic to the quotient space 
\[
\frac{(\partial N\times[0,1])\sqcup K}{(x,1)\sim r(x)}.
\]
This ensures that $K$ is a topological spine and a strong deformation retract of $N$. Consider the long exact sequence of singular homology of the pair $(N,\partial N)$.
\begin{equation}
\cdots\longrightarrow H_3(N;\mathbb{Z}) \longrightarrow H_3(N,\partial N;\mathbb{Z})\longrightarrow H_2(\partial N;\mathbb{Z})\longrightarrow H_2(N;\mathbb{Z})\longrightarrow\cdots
\label{eq:exact}
\end{equation}
Since $N$ and $K$ have the same homotopy type it follows that $H_3(N;\mathbb{Z})\cong H_2(N;\mathbb{Z})\cong\{0\}$. Taking this and \eqref{eq:exact} into account we get that $H_3(N,\partial N;\mathbb{Z})\cong H_2(\partial N;\mathbb{Z})$. As a consequence $\partial N$ is orientable if and only if $N$ is orientable. We see that $\partial N$ is connected. Consider the terminal part of the long exact sequence of reduced homology with $\mathbb{Z}_2$ coefficients of the pair $(N,\partial N)$. 
\begin{equation}
\cdots\longrightarrow H_1(N,\partial N;\mathbb{Z}_2) \longrightarrow \widetilde{H}_0(\partial N;\mathbb{Z}_2)\longrightarrow \widetilde{H}_0(N;\mathbb{Z}_2)\longrightarrow \widetilde{H}_0(N,\partial N;\mathbb{Z}_2)\longrightarrow 0
\label{eq:exact1}
\end{equation}
Since $N$ is connected $\widetilde{H}_0(N;\mathbb{Z}_2)\cong\{0\}$. On the other hand, Lefschetz duality Theorem ensures that $H_1(N,\partial N;\mathbb{Z}_2)\cong H^2(N;\mathbb{Z}_2)\cong\{0\}$. Taking this and \eqref{eq:exact1} into account it follows that $\widetilde{H}_0(\partial N;\mathbb{Z}_2)\cong\{0\}$ and, hence, $\partial N$ is connected.  A new application of Lefschetz duality gives that $\chi(N,\partial N)=-\chi(N)$ and, since $\chi(N,\partial N)=\chi(N)-\chi(\partial N)$, it follows that $\chi(\partial N)=2\chi(N)=0$. Therefore $\partial N$ is a connected closed $2$-manifold with zero Euler characteristic. Thus $\partial N$ is homeomorphic to the torus if $N$ is orientable and to the Klein bottle if $N$ is not. It remains to see that $N$ is orientable if and only if $K$ preserves orientation. Suppose that $N$ is orientable. Then the open $3$-manifold ${\rm int}(N)$ is orientable and \cite[Lemma~3.27, pg. 236]{Hat} ensures the existence of a homology class $\alpha_K\in H_3({\rm int}(N),{\rm int}(N)\setminus K;\mathbb{Z})$ such that $i_{x*}(\alpha_K)$ is a generator of $H_3({\rm int}(N),{\rm int}(N)\setminus \{x\};\mathbb{Z})$ for each $x\in K$, where $i_x:({\rm int}(N),{\rm int}(N)\setminus K)\hookrightarrow  ({\rm int}(N),{\rm int}(N)\setminus\{x\})$ denotes the inclusion map. Let $k:{\rm int}(N)\hookrightarrow M$ and $j_x: (M,M\setminus K)\hookrightarrow (M,M\setminus\{x\})$ be the inclusions. Then for each $x\in K$ we get the following commutative diagram 
\begin{equation}
\begin{CD}
H_3({\rm int}(N),{\rm int}(N)\setminus K;\mathbb{Z})@>k_*>> H_3(M,M\setminus K;\mathbb{Z})\\
@VVi_{x*}V     @VVj_{x*}V \\
H_3({\rm int}(N),{\rm int}(N)\setminus \{x\};\mathbb{Z}) @>k_*>> H_3(M,M\setminus \{x\};\mathbb{Z})
\end{CD}
\label{eq:inclusions}
\end{equation}
The excision property of homology ensures that the homomorphisms $k_*$ in \eqref{eq:inclusions} are isomorphisms. As a consequence, the homology class $\beta_K=k_*(\alpha_K)\in H_3(M,M\setminus K;\mathbb{Z})$ satisfies that $j_{x*}(\beta_K)$ is a generator of $H_3(M,M\setminus\{x\};\mathbb{Z})$ for each $x\in K$. Therefore $K$ preserves orientation.

Suppose that $K$ preserves orientation. Then $H_3(M,M\setminus K;\mathbb{Z})\neq\{0\}$ and excision ensures that $H_3({\rm int}(N),{\rm int}(N)\setminus K;\mathbb{Z})\neq\{0\}$. On the other hand, since $\partial N$ is collared in $N$ and $K$ is a spine of $N$ it easily follows that $H_3(N,\partial N;\mathbb{Z})\cong H_3({\rm int}(N),{\rm int}(N)\setminus K;\mathbb{Z})$ and, hence, nontrivial. Therefore $N$ is orientable.

The result follows from the Main Theorem. 
\end{proof}

  We recall that a connected closed $2$-manifold $S$ contained in a $3$-manifold $M$ is \emph{two sided} in $M$ if $S$ separates every sufficiently small connected neighborhood of itself in $M$. More precisely, $S$ is two sided in $M$ if there exists a neighborhood $V$ of $S$ in $M$ such that if $W$ is a connected neighborhood of $S$ in $M$ with $W\subset V$, then $W\setminus S$ is not connected. In general we say that a (possibly not connected) closed $2$-manifold $S$ contained in a $3$-manifold $M$  is two sided in $M$ if every component of $S$ is two sided in $M$.

\begin{theorem}\label{thm:closed}
Suppose that $M$ is a $3$-manifold and $K\subset M$ is a tame closed $2$-manifold that is two sided in $M$. Suppose that $K$ is an attractor for $f_0$  that  undergoes a Hopf bifurcation at $\lambda=0$. Then, there exists $\lambda_0>0$ such that for every $\lambda$ with $0<\lambda<\lambda_0$ there exists an attractor $K_\lambda$ that has the Borsuk's homotopy type of two disjoint copies of $K$ and converges to $K$ upper semicontinuosly as $\lambda$ tends to $0$. 
\end{theorem}

\begin{proof}
Since $K\subset M$ is a tame closed $2$-manifold that is two sided in $M$  \cite[Theorem~3, pg. 192]{Moiseb} ensures that $K$ is bicollared. Let $C$ be a bicollar of $K$. Then there exists a homeomorphism $h:C\longrightarrow K\times[0,1]$ such that $h(K)=K\times\{0\}$. Hence $K$ is a topological spine of $C$ and the result follows by applying the Main Theorem.
\end{proof}

\begin{theorem}\label{thm:sub}
Suppose that $M$ is a $3$-manifold and $K\subset M$ is a tame compact $3$-manifold with boundary that is an  attractor for $f_0$. Suppose that $K$ undergoes a Hopf bifurcation at $\lambda=0$. Then, there exists $\lambda_0>0$ such that for every $\lambda$ with $0<\lambda<\lambda_0$ there exists an attractor $K_\lambda$ that has the Borsuk's homotopy type of $\partial K$  and converges to $K$ upper semicontinuosly as $\lambda$ tends to $0$. 
\end{theorem}

\begin{proof}
Since $K\subset M$ is a tame compact $3$-manifold with boundary $\partial K$ is bicollared \cite[Theorem~2, pg. 191]{Moiseb}. Let $C$ be a bicollar of $\partial K$ and consider the compact $3$-manifold with boundary $N=K\cup C$. The construction ensures that $K\subset{\rm int}(N)$ and that $K$ and $N$ are concentric. Therefore $K$ is a topological spine of $N$ and the result follows by applying the Main Theorem.
\end{proof}

\section{A characterization of attractors}\label{sec:5}

In this section we present some results regarding attractors in locally compact metric spaces that will be useful in the sequel. These results are well known in the case of flows (see \cite[Lemma~3.1]{Sal} and \cite[Lemma~3.2]{Sal}). They are probably also well-known in general but since we did not find any reference about them in the discrete case we state and prove them here. 

Let $X$ be a locally compact metric space and $f:X\longrightarrow X$ a homeomorphism. An invariant compactum $K\subset X$ is said to be \emph{isolated} if  it is the maximal invariant set contained in a compact neighboord $N$ of itself. Such a neighborhood $N$ is said to be an \emph{isolated neighborhood}. In particular, attractors are isolated invariant. sets. Let $N$ be a compactum that is the closure of its interior. We define the \emph{exit set} of $N$ as
\[
N^-=\{x\in N\mid f(x)\notin \mathring{N}\}.
\]
Let $K$ be an isolated invariant set. A compact pair $(N,L)$ is said to be a \emph{filtration pair} if it satisfies
\begin{enumerate}
\item[(1)] $N$ and $L$ are the closure of their interiors.
\item[(2)] $\overline{N\setminus L}$ is an isolated neighborhood of $K$.
\item[(3)] $L$ is a neighborhood of $N^-$ in $N$.
\item[(4)] $f(L)\cap \overline{N\setminus L}=\emptyset$.
\end{enumerate}

Filtration pairs were introduced and studied by Franks and Richeson in \cite{FrRic}. A useful fact about filtration pairs is that every isolated invariant set has a basis of neighborhoods comprised of filtration pairs \cite[Proposition~3.5]{FrRic} and \cite[Theorem~3.6]{FrRic}.

The following result can be regarded as a discrete version of \cite[Lemma~3.1]{Sal}. The proof for the discrete case is an adaptation of the proof for flows using filtration pairs to overcome the lack of continuous trajectories.

\begin{proposition}\label{at1}
Let $X$ be a locally compact metric space and $f:X\longrightarrow X$ a homeomorphism. Then, an invariant compactum $K\subset X$ is an attractor if and only if there exists a neighborhood $U$ of $K$ such that for every $x\in U\setminus K$ there exists some $n_x>0$ such that $f^{-n_x}(x)\notin U$.
\end{proposition}

\begin{proof}
Suppose that $K$ is an attractor and let $U$ be a compact neighborhood such that $K=\omega(U)$. Assume that $x\in U$ such that $f^{-n}(x)\in U$ for every $n\geq 0$. Then, $x\in f^n(U)$ for every $n\geq 0$ and, as a consequence, $x\in \omega(U)=K$. This proves the necessity.

Conversely, let $U$ be a neighborhood of $K$ satisfying that for every $x\in U\setminus K$ there is some $n_x>0$ such that $f^{-n_x}(x)\notin U$. The existence of such neighborhood ensures that $K$ is an isolated invariant set. Let $(N,L)$ be a filtration pair for $K$ contained in $U$. Then, the condition on $U$ ensures that for each $x\in N\setminus K$ there exists some $m_x>0$ such that $f^{-m_x}(x)\notin N$. Since $L$ is compact there exists $n_0>0$ such that for every point $x\in L$, $\cup_{0\leq n\leq n_0} f^{-n}(x)\nsubseteq N$. Now choose a neighborhood $V\subset N$ of $K$ disjoint from $L$ and with the additional property that $f^n(V)\subset N$ for every $n$ with $0\leq n\leq n_0$. Then $f^n(V)\subset N$ for every $n\geq 0$. Indeed, suppose on the contrary that there exists a point $x\in V$ and $n'>0$ such that $f^{n'}(x)\notin N$. Taking into account that $L$ is a neighborhood of $N^-$ choose $n_L>0$ such that $f^n(x)\in N$ for every $n$ with $0\leq n\leq n_L$, $f^{n_L}(x)\in L$ and $f^{n_L+1}(x)\notin N$. Notice that, by assumption, $n_L\geq n_0$.  Since $f^{n_L}(x)\in L$ there exists $k$ with $0\leq k\leq n_0$ such that $f^{n_L-k}(x)\notin N$.  This is in contradiction with the choice of $n_L$. Hence, $\omega(V)$ is a compact invariant set contained in $N$ with $K\subset \omega(V)$ and, since $N$ is an isolating neigborhood of $K$ the equality follows.
\end{proof}

\begin{proposition}\label{at2}
Let $X$ be a locally compact metric space and let $f:X\rightarrow X$ be a
homeomorphism. Suppose that $K$ is an attractor of $f$ and that $
K_{0}\subset K$ is an attractor of $f|_K$. Then $K_{0}$ is an attractor of $f$.
\end{proposition}

\begin{proof}
Since $K$ is an attractor of $f$ and $K_0$ is an attractor of $f|_K$, Proposition~\ref{at1} ensures the existence of neighborhoods $V$ and $W$ of $K$ and $K_0$ respectively such that for every $x\in V\setminus K$ there exists $n_x>0$ such that $f^{-n_x}(x)\notin V$ and for every $x\in (W\cap K)\setminus K_0$ there exists $m_x>0$ such that $f^{-m_x}(x)\notin W$. Then,  $U=V\cap W$ is a neighborhood of $K_0$ and, if $x\in U\setminus K_0$ we have two possibilites, either $x\in U\setminus K$ and $f^{-n_x}(x)\notin U$ or $x\in U\cap K\setminus K_0\subset (W\cap K)\setminus K_0$ and $f^{-m_x}(x)\notin U$. Therefore, Proposition~\ref{at1} ensures that $K_0$ is an attractor of $f$. 
\end{proof}

\section{Proof of the Main Theorem}\label{sec:6}

Let $W$ be a trapping region of $K$ for $f_0$. By assumption $K$ is a topological spine of some compact $n$-manifold  with boundary $N$ contained in $M$. By Proposition~\ref{prop:con} we may assume that $N\subset\mathring{W}$ and that $\partial N$ is bicollared. We recall that the Theorem of invariance of domain guarantees that ${\rm int}(N)$ coincides with the topological interior of $N$. Since $W$ is a trapping region of $K$ there exists $k>0$ such that $f_0^k(W)\subset{\rm int}(N)$. Then, the continuous dependence on the parameter ensures the existence of $\lambda_0>0$ such that $W$ is a trapping region for $f_\lambda$  for $0<\lambda<\lambda_0$ satisfying that 
\begin{equation}
f^k_\lambda(W)\subset{\rm int}(N).
\label{eq:1}
\end{equation}
 Let $S_\lambda$ be the attractor of $f_\lambda$ determined by $W$. Since $K$ is  an invariant set for $f_\lambda$ contained in $W$ it follows that $K\subset S_\lambda$. Moreover,  the attracting character of $S_\lambda$ ensures that the basin of repulsion $\mathcal{R}_\lambda$ of $K$ is also contained in $S_\lambda$. Then
 \begin{equation}
 K_\lambda=S_\lambda\setminus \mathcal{R}_\lambda 
 \label{eq:2}
 \end{equation}
is an attractor for $f_\lambda|_{S_\lambda}$ and, since $S_\lambda$ is an attractor, Proposition~\ref{at2} ensures that $K_\lambda$ is an attractor for $f_\lambda$. 

 We prove that $K_\lambda$ has the Borsuk's homotopy type of $\partial N$. Since $\mathcal{R}_\lambda$ is an open set such that $K\subset\mathcal{R}_\lambda\subset N$ by Proposition~\ref{prop:con} there exists a compact $n$-manifold with boundary $N_1\subset\mathcal{R}_\lambda$ that is concentric with $N$ and such that $K\subset{\rm int}(N_1)$. Let $A_0=N\setminus{\rm int}(N_1)$ be the compact $n$-manifold with boundary bounded by $\partial N$ and $\partial N_1$. The manifold $A_0$ is a compact neighborhood of $K_\lambda$. Moreover, since the basin of attraction of $K_\lambda$ is $\mathcal{A}_\lambda\setminus K$ where $\mathcal{A}_\lambda$ is the basin of attraction of $S_\lambda$, it follows that $A_0$ is contained in the basin of attraction of $K_\lambda$. As a consequence there exists $j>0$ such that $f_\lambda^j(A_0)\subset{\rm int}(A_0)$. If for every $m>0$ we choose $A_m=f_\lambda^{mj}(A_0)$  we get a sequence of compact $n$-manifolds with boundary such that $A_{m+1}\subset {\rm int}(A_m)$ and
 \begin{equation}
 K_\lambda=\bigcap_{m\geq 0} A_m.
 \label{eq:3}
 \end{equation}

We see that $A_0$ and $A_1$ are concentric. Since $f^j_\lambda$ is a homeomorphism and $A_1=f^j_\lambda(A_0)$ it follows that 
\begin{equation}
{\rm int}(A_1)=f^j_\lambda({\rm int}(A_0))=f^j_\lambda({\rm int}(N)\setminus N_1)=f^j_\lambda({\rm int}(N))\setminus f^j_\lambda(N_1).
\label{eq:4}
\end{equation}
Then
\begin{equation}
A_0\setminus{\rm int}(A_1)=((N\setminus {\rm int}(N_1)\setminus f^j_\lambda({\rm int}(N)))\cup ((N\setminus {\rm int}(N_1))\cap f^j_\lambda(N_1)).
\label{eq:5}
\end{equation}
Since $N_1\subset f^j_\lambda(N)\subset N$ and $A_1\subset {\rm int}(A_0)$ it easily follows that $N_1\subset {\rm int}(f^j_\lambda(N_1))\subset f^j_\lambda(N_1)\subset N$. Taking this into account together with the fact that $N_1\subset f^j_\lambda(N)$ in \eqref{eq:5} yields
\begin{equation}
A_0\setminus{\rm int}(A_1)=(N\setminus f^j_\lambda({\rm int}(N)))\cup (f^j_\lambda(N_1)\setminus{\rm int}(N_1)).
\label{eq:6}
\end{equation}
 
The compact $n$-manifold with boundary $f_\lambda^j(N)\subset {\rm int}(N)$ has bicollared boundary homeomorphic to $\partial N$ and satisfies that $N_1\subset{\rm int}(f_\lambda^j(N))$. Hence, since $N_1$ is concentric with $N$ and $N$ has the concentric rigidity property it follows that $f_\lambda^j(N)$ is concentric with $N$. Reasoning in the same fashion with $f_\lambda^j(N_1)$  we get that $f_\lambda^j(N_1)$ and $N_1$ are concentric. The concentricity of $A_0$ and $A_1$ follows from this discussion after noticing that the fact that $N_1\subset{\rm int}(N)$ ensures that the union in \eqref{eq:6} is disjoint.

Since $A_0$ and $A_1$ are concentric and for $m>0$ the map
\[
f_\lambda^{mj}|_{A_0}:(A_0,A_1)\longrightarrow (A_m,A_{m+1})
\]
is a homeomorphism we obtain that $A_m$ is concentric with $A_{m+1}$ for every $m\geq 0$. As a consequence $K_\lambda$ is a topological spine of $A_0$ and, hence, Corollary~\ref{cor:shape} guarantees that the inclusion $i:K_\lambda\hookrightarrow A_0$ is a shape equivalence. Hence, $K_\lambda$ has the Borsuk's homotopy type of $A_0$. However, the concentricity of $N$ and $N_1$ ensures that $A_0$ is homeomorphic to $\partial N\times[0,1]$ and, thus, has the homotopy type of $\partial N$. Therefore $K_\lambda$ has the shape of $\partial N$.

To finish the proof we observe  that \eqref{eq:1} ensures that $K_\lambda\subset{\rm int}(N)$. Since $N$ can be chosen arbitrarily close to $K$, it follows that $K_\lambda$ converges to $K$ upper-semicontinuously as $\lambda\to 0$. \qed

\section{A generalization:\ attracting spheres of lower dimensions.}\label{sec:7}

We have seen in Theorem 11 that when an asymptotically stable point $p$
undergoes a Hopf bifurcation in an $n$-dimensional manifold then a family of
attractors with the Borsuk homotopy type of $S^{n-1}$evolves from $p$ .
There are, however, other bifurcations in which, under partially weaker
conditions, attractors are still produced with the Borsuk homotopy type of $%
S^{k}$ for some $k<n-1$. Consider for instance the discrete system in $%
\mathbb{R}^{3}$ corresponding to the time $1$ function of the Lorenz flow 
depending on the classical parameter $r$ (the Rayleigh number) \cite{Lor}. It is known
that for $r=1$ a pitchfork bifurcation takes place at the origin which
creates attractors which are $0$-dimensional spheres. This is only a
particular case of more general pitchfork bifurcations in $\mathbb{R}^{n}$
which produce attracting spheres $S^{k}$ for $k<n-1$ evolving from equilibria. In the following proposition we present a result which encompasses many
of these situations.

\begin{theorem}
Let $M$ be a $d$-dimensional manifold and $f_{\lambda }:M\rightarrow M$ be
a family of homeomorphisms continuously depending on $\lambda \in I$. Let $N$
be an $n$-dimensional submanifold which is invariant for all $\lambda$.
Suppose that the point $p\in N$ is an attractor for $f_{0}$ and a repeller
for the restriction $f_{\lambda }|_N$ for $\lambda >0$. Suppose, additionaly,
that there exists a neighborhood $V$ of $p$ in $M$ (the same for all $
\lambda $) such that the maximal invariant set of $f_{\lambda }$ inside $V$
is contained in $N$. Then there exists $\lambda _{0}>0$ such that for every $%
\lambda $ with $0<\lambda <\lambda _{0}$ there exists an attractor $%
K_{\lambda }$ of $f_{\lambda }$ which has the Borsuk homotopy type of $%
S^{n-1}$. Moreover the family $\{K_{\lambda }\}$ converges to $\{p\}$ upper
semicontinuously as $\lambda $ tends to $0$.
\end{theorem}

\begin{remark}
Notice that the condition in the statement concerns only the restriction $%
f_{\lambda }|_N$ but the conclusion is that $K_{\lambda }$ is an attractor of 
$f_{\lambda }$ (not only of $f_{\lambda }|_N$).
\end{remark}

\begin{proof}
The proof consists of an adaptation of the proofs of previous results to the
conditions of the present theorem.

Let $W\subset V$ be a trapping region of $p$ for $f_{0}$. Consider a
topological $n$-cell $C\subset N\cap W$ such that $p\in{\rm int}(C)$ and $%
\partial C$ is bicollared in $N$. Let $\hat{C}$ be a neighborhood of $p$ in $%
M $ such that $\hat{C}\cap N\subset {\rm int}(C)$. Since $W$ is a trapping
region of $p$ there exists $k_0\geq 0$ such that $f_{0}^{k}(W)\subset 
\hat{C}$ for $k\geq k_0$. Moreover, the continuous dependence on the parameter ensures the
existence of $\lambda _{0}>0$ such that $W$ is a trapping region of $p$ and $%
f_{\lambda }^{k}(W)\subset \hat{C}$ for $\lambda \leq \lambda _{0}$ and $k\geq k_0$.
Consider the attractor $S_{\lambda }$ of $f_\lambda$ determined by $W$. Since for every $\lambda\in I$ the maximal invariant set inside $V$ is contained in $N$, the previous discussion ensures that $p\in S_{\lambda
}\subset {\rm int}(C)$. Moreover, since $S_{\lambda }$ is an attractor, the basin of
repulsion $\mathcal{R}_{\lambda }$ of $p$ for $f_{\lambda }|_N$ is contained
in $S_{\lambda }$. Consider $K_{\lambda }=S_{\lambda }\setminus\mathcal{R}_{\lambda
} $. Then $K_{\lambda }$ is an attractor of $f_{\lambda }|_{S_{\lambda}}$ and,
since $S_{\lambda}$ is an attractor of $f_{\lambda }$, then Proposition~\ref{at2} ensures that $K_{\lambda }$
is also an attractor of $f_{\lambda }$ (not just of $f_{\lambda }|_N$).

We must now prove that $S_{\lambda}$ has the Borsuk homotopy type of $%
\partial C\approx S^{n-1}$. In the sequel we consider the restriction $%
f_{\lambda }|_N$. Take an $n$-cell $C_{1}\subset \mathcal{R}_{\lambda }$
concentric with $C$ and such that $p\in {\rm int}(C_{1})$. Let $A_{0}=C\setminus{\rm int}(
C_{1})$. Obviously, $A_{0}$ is a compact neighborhood of $K_{\lambda }$ in $%
N $ contained in the basin of attraction of $K_{\lambda }$. As a
consequence, there exists $j>0$ such that $f_{\lambda }^{j}(A_{0})\subset 
{\rm int}(A_{0})$. Now consider $A_{m}=f_{\lambda }^{mj}(A_{0})$ and in this way
we obtain a sequence of compact $n$-manifolds with boundary such that $%
A_{m+1}\subset {\rm int}(A_{m})$ and $K_{\lambda}=\cap_{m\geq 0} A_{m}$. The same argument as in
the proof of the Main Theorem proves that $K_{\lambda }$ is a topological spine of $A_{0}$
and, therefore, the inclusion $i:K_{\lambda }\hookrightarrow A_{0}$ is a shape
equivalence. Moreover, since $C_1$ has been chosen to be concentric with $C$, $A_{0}$ is
homeomorphic to $S^{n-1}\times [0.1]$, and, thus, $A_{0}$ is homotopy
equivalent to $S^{n-1}$. We conclude that $K_{\lambda }$ is shape equivalent
to $S^{n-1}$. The rest of the proof is straightforward.
\end{proof}

\section*{Acknowledgment}

The authors thank Jos\'e M. Montesinos-Amilibia and Jaime J. S\'anchez-Gabites for inspiring conversations.

\bibliography{Biblio1}        %use a bibtex bibliography file refs.bib
\bibliographystyle{plain}  %use the plain bibliography style

\end{document}